\documentclass[10pt,reqno]{amsart}
\usepackage{graphicx, amssymb, color,pdfsync}
\usepackage[all,cmtip]{xy}
\numberwithin{equation}{section}
\usepackage{mathrsfs}
\usepackage{tikz}
\usetikzlibrary{matrix,arrows}
\usepackage{dsfont} 
\usepackage[USenglish]{babel}
\usepackage{tikz} 

\begin{document} 
\newcommand{\s}{\vspace{0.2cm}} 

\newtheorem{theo}{Theorem} 
\newtheorem{prop}{Proposition}
\newtheorem{coro}{Corollary}
\newtheorem{lemm}{Lemma}
\newtheorem{example}{Example}
\theoremstyle{remark}
\newtheorem{rema}{\bf Remark} 
\newtheorem{defi}{\bf Definition}

\newcommand{\rac}{{\mathbb{Q}}}
\newcommand{\comp}{{\mathbb{C}}}
\newcommand{\hip}{{\mathbb{H}}}
\newcommand{\hola}{{\overline{\rac}}}

\title[Weil's Galois Descent Theorem]{Weil's Galois Descent Theorem from a computational point of view}
\author{Rub\'en A. Hidalgo and Sebasti\'an Reyes-Carocca}
\address{Departamento de Matem\'aticas y Estad\'istica, Universidad de La Frontera, Avenida Francisco Salazar 01145, Temuco, Chile.} 
\email{ruben.hidalgo@ufrontera.cl, sebastian.reyes@ufrontera.cl} 
\thanks{The first author was partially supported by Fondecyt  Grant 1150003. The second author was partially supported by Fondecyt Grant 11180024, 1190991 and Redes Grant 2017-170071}
\keywords{Algebraic varieties, Galois extensions, Fields of definition}
\subjclass[2010]{14E99, 14A10, 12F10}

\begin{abstract}
Let ${\mathcal L}/{\mathcal K}$ be a finite Galois extension and let $X$ be an affine algebraic variety defined over ${\mathcal L}$. Weil's Galois descent theorem provides  necessary and sufficient conditions for  $X$ to be definable over ${\mathcal K}$, that is, for the existence of an algebraic variety $Y$ defined over ${\mathcal K}$ together with a birational isomorphism $R:X \to Y$ defined over ${\mathcal L}$. Weil's proof does not provide a method to construct the birational isomorphism $R.$ The aim of this paper is to give an explicit construction of $R$.
\end{abstract}
\maketitle

\section{Introduction}
Let ${\mathcal K}$ be a perfect field and let ${\mathcal C}$ be an algebraic closure of it.  An affine  algebraic  variety $X \subset {\mathcal C}^{n}$ is said to be {\it defined} over a subfield ${\mathcal R}$ of ${\mathcal C}$  if  its corresponding ideal of polynomials $\mathbb{I}(X)$ can be generated by a  finite collection of polynomials with coefficients in ${\mathcal R}$. Let us assume that  $X$ is defined over a subfield ${\mathcal L}$ of  ${\mathcal C}$ which is a finite Galois extension of ${\mathcal K}$. Under this assumption, we will say that $X$ is {\it definable} over ${\mathcal K}$, with respect to the Galois extension ${\mathcal L}/{\mathcal K}$,  if there is an algebraic variety $Y$ defined over ${\mathcal K}$ and a birational isomorphism $R:X \to Y$ defined over ${\mathcal L}$.

To decide whether or not $X$ is definable over ${\mathcal K}$  is, in general, a difficult task. For instance, if ${\mathcal L}$ is the field of complex numbers and ${\mathcal K}$ is the field of real numbers, then there are known explicit examples of complex algebraic curves which are not definable over the reals. These examples were provided by Shimura \cite{Shimura} and Earle \cite{Earle,Earle2} and later by Huggins \cite{Huggins} (in the hyperelliptic case) and by the first author \cite{Hidalgo} and Kontogeorgis \cite{Ko} (in the non-hyperelliptic situation).

The natural action of $\mbox{Gal}({\mathcal L}/{\mathcal K})$ on  the ring of polynomials with coefficients in $\mathcal{L}$ induces a well-defined action $(\sigma, X)\to X^{\sigma}$ on the set of birational isomorphism classes of algebraic varieties. A collection of birational isomorphisms  
$$\{f_{\sigma}:X \to X^{\sigma}: \sigma \in \mbox{Gal}({\mathcal L}/{\mathcal K})\}$$ defined over ${\mathcal L}$ satisfying the so-called {\it Weil's co-cycle condition} $$f_{\tau\sigma}=f_{\sigma}^{\tau} \circ f_{\tau} \,\, \mbox{ for each } \,\, \sigma,\tau \in {\rm Gal}({\mathcal L}/{\mathcal K})$$ is called a {\it Galois descent datum} for $X$ with respect to ${\mathcal L}/{\mathcal K}$. 
 
 Assume that $X$ is definable over $\mathcal{K}.$ Namely, suppose the existence of a birational isomorphism $R:X \to Y$ defined over ${\mathcal L}$ where $Y$ is an affine algebraic variety  defined over ${\mathcal K}$. Then $Y=Y^{\sigma}$ for each $\sigma \in {\rm Gal}({\mathcal L}/{\mathcal K}),$ and the collection $$\{(R^{\sigma})^{-1} \circ R:X \to X^{\sigma}: \sigma \in {\rm Gal}({\mathcal L}/{\mathcal K})\}$$ is a Galois descent datum for $X$ with respect to ${\mathcal L}/{\mathcal K}.$ In other words, the existence of such a Galois descent datum for $X$ is a necessary condition for $X$ to be definable over ${\mathcal K}$.
Conversely, Weil in \cite{Weil} proved that the existence of such a Galois descent datum is also sufficient condition. More precisely:

\s

{\bf Weil's Galois  descent theorem.} 
Let ${\mathcal K}$ be a perfect field, let ${\mathcal C}$ be an algebraic closure of ${\mathcal K}$ and let ${\mathcal L}$ be a subfield of ${\mathcal C}$ that is a finite Galois extension of ${\mathcal K}.$ Set $\Gamma={\rm Gal}({\mathcal L}/{\mathcal K})$ and assume that $X$ is an affine algebraic variety defined over ${\mathcal L}$.

\begin{enumerate}
\item[(a)] If $X$ admits a Galois descent datum $\{f_{\sigma}\}_{\sigma \in \Gamma}$ with respect to  ${\mathcal L}/{\mathcal K}$, then there exists an algebraic variety $Y$, defined over ${\mathcal K}$, and there exists a birational isomorphism $R:X \to Y$, defined over ${\mathcal L}$, such that $R=R^{\sigma} \circ f_{\sigma}$ for every $\sigma \in \Gamma$. Moreover, if all the  isomorphisms $f_{\sigma}$ are biregular then $R$ can be chosen to be biregular. 

\item[(b)] If there is another birational isomorphism  $\hat{R}:X \to \hat{Y}$, defined over ${\mathcal L}$, where $\hat{Y}$ is defined over ${\mathcal K}$, such that $\hat{R}=\hat{R}^{\sigma} \circ f_{\sigma}$ for every $\sigma \in \Gamma$, then there exists a birational isomorphism $J:Y \to \hat{Y}$, defined over ${\mathcal K}$, such that $\hat{R}=J \circ R$.
\end{enumerate}

\s

Weil's proof does not provide an algorithm to construct the isomorphism $R$ explicitly. However, in the proof of \cite[Proposition 1]{Weil}, it was observed that if $\{f_{\sigma}\}$ is a Galois descent datum, if each  $f_{\sigma}$ is biregular and if an explicitly  birational map $R$ as before is known, then there is an explicit method to obtain a new biregular isomorphism $X \to Z$, defined over ${\mathcal L}$, with $Z$ still defined over ${\mathcal K}$.  Such an explicit method is given by considering the map 
$$F:X \to Y^{n} \,\, \mbox{ defined by } \,\, x \mapsto (R(x),R^{\sigma_{2}}(x),\ldots,R^{\sigma_{n}}(x)),$$
where 
$\Gamma=\{\sigma_{1}=e,\sigma_{2},\ldots,\sigma_{n}\}$ and then, as there is a natural permutation action of $\Gamma,$ to consider the classical invariant theory to construct a regular map $\Psi:Y^{n} \to Z$, defined over ${\mathcal L}$, so that $Z$ is defined over ${\mathcal K}$ and $\Psi \circ F:X \to Z$ is a biregular isomorphism.  

\s

In this article, we follow similar ideas as above to construct explicitly a rational map $R:X \to {\mathcal C}^{m}$, defined over ${\mathcal L}$, such that
$Y=R(X)$ is defined over ${\mathcal K}$ and $R:X \to Y$ is a birational isomorphism. This explicit construction is done in terms of  equations for $X$ and of a Galois descent datum for $X$ with respect to ${\mathcal L}/{\mathcal K}$. Since $R$ is explicitly given, the algorithm can be used to compute explicit equations for $Y$ over ${\mathcal K}$. Indeed, for the sake of completeness,  in the last section we will work out an  example where $X$ is a complex algebraic curve  of genus five defined over $\mathbb{Q}(i)$. This curve admits  a  group of conformal automorphisms isomorphic to ${\mathbb Z}_{2}^{4}$  and it is also endowed with an anticonformal involution; so, it is definable over $\mathbb{Q}$.

\section{Preliminaries}

Let ${\mathcal K}$ be a perfect field, let ${\mathcal C}$ be an algebraic closure of it, and let ${\mathcal L}$ be a subfield of ${\mathcal C}$ which is a finite Galois extension of ${\mathcal K}.$ We denote by $$\hat{\Gamma}={\rm Gal}({\mathcal C}/{\mathcal K}) \, \, \mbox{ and } \, \, \Gamma={\rm Gal}({\mathcal L}/{\mathcal K})$$the Galois group associated to the  extensions ${\mathcal C}/{\mathcal K}$ and ${\mathcal L}/{\mathcal K}$ respectively.

Each $\eta \in\hat{\Gamma}$  induces a natural bijection
$$\hat{\eta}:{{\mathcal C}}^{n} \to {{\mathcal C}}^{n} \,\, \mbox{ given by } (y_{1},\ldots, y_{n}) \mapsto (\eta(y_{1}), \ldots, \eta(y_{n})),$$and if $P \in {\mathcal C}[z_{1},\dots,z_{n}]$ then we denote by $P^{\eta}$ the polynomial obtained after applying $\eta$ to the coefficients of $P$. In other words, the following diagram commutes.
%
%In other words $P^{\eta}=\widehat{\eta} \circ P \circ \widehat{\eta}^{-1}$, where $\widehat{\eta}$ at the right acts on ${\mathcal C}^{n}$ and the one at the left acts on ${\mathcal C},$ as shown in the diagram below.
\begin{equation*}
\begin{tikzpicture}[node distance=3.1 cm, auto]
  \node (P) {$\mathcal{C}^n$};
  \node (Q) [right of=P] {$\mathcal{C}$};
  \node (A) [below of=P, node distance=1.2 cm] {$\mathcal{C}^n$};
  \node (C) [below of=Q, node distance=1.2 cm] {$\mathcal{C}$};
  \draw[->] (P) to node {$P$} (Q);
  \draw[->] (A) to node {$P^{\eta}$} (C);
  \draw[->] (P) to node [swap] {$\hat{\eta}$} (A);
    \draw[->] (Q) to node {${\eta}$} (C);
\end{tikzpicture}
\end{equation*}

Let $$X = \{ (y_{1},\ldots, y_{n})  \in  {\mathcal C}^{n} : P_j(y_{1},\ldots, y_{n})=0, \, 1 \le j \le r \}$$be an affine algebraic variety where each $P_j \in {\mathcal L}[z_1, \ldots, z_n].$ If $\eta \in \hat{\Gamma}$ then $P_j^{\eta} \in {\mathcal L}[z_{1},\ldots,z_{n}]$ and 
$$\hat{\eta}(X) = \{ (y_{1},\ldots, y_{n})  \in  {\mathcal C}^{n} : P_j^{\eta}(y_{1},\ldots, y_{n})=0, \, 1 \le j \le r \}.$$

Let us denote by $\rho:\hat{\Gamma} \to \Gamma$  the canonical epimorphism defined by restriction. Note that:
\begin{enumerate} 
\item if $\rho(\eta)=\sigma$ then $P_j^{\eta}=P_j^{\sigma},$ and \item if  $\rho(\eta_{1})=\rho(\eta_{2})$ then $\hat{\eta_{1}}(X)=\hat{\eta_{2}}(X)$. 
\end{enumerate} 

Then, if $\rho(\eta)=\sigma$ then we denote $\hat{\eta}(X)$ by  $X^{\sigma}$. 

%In general it may happen that $X^{\sigma}$ and $X$ are not birationally equivalent.

\s

Let $\{e_{1},  \ldots, e_{m}\}$ be a basis of ${\mathcal L}$ as a ${\mathcal K}-$vector space. Then the matrix 
$$A=\left[ \begin{array}{cccc}
e_{1} & e_{2} & \cdots & e_{m}\\
\sigma_{2}(e_{1}) & \sigma_{2}(e_{2}) & \cdots & \sigma_{2}(e_{m})\\
\vdots & \vdots & \ddots & \vdots \\
 \sigma_{m}(e_{1}) & \sigma_{m}(e_{2}) & \cdots & \sigma_{m}(e_{m})
\end{array}
\right] \in \mbox{M}(m \times m, \mathcal{L})
$$is non-singular (see, for example, \cite{Hungerford}).  %
%The following fact about the linearly independency of automorphisms  of Galois extensions can be found in Hungenford's book \cite{Hungerford}.
%
%
%\begin{lemm}[\cite{Hungerford}]
%The matrix $A$ has non-zero determinant.
%\end{lemm}
%
The trace map 
$${\rm Tr}:{\mathcal L} \to {\mathcal K} \,\, \mbox{ given by } \,\, a \mapsto \Sigma_{j=1}^{m} \sigma_{j}(a)$$ extends naturally to polynomial rings
$${\rm Tr}:{\mathcal L}[x_{1}, \ldots ,x_{n}] \to {\mathcal K}[x_{1}, \ldots ,x_{n}] \,\, \mbox{ given by } \,\, P \mapsto \Sigma_{j=1}^{m} P^{\sigma_{j}}.$$ 
%
%A direct consequence of Lemma \ref{invertibilidad} is the following fact about invariant ideals.
%
%
\begin{lemm}\label{polinv}
Under the above notations, we have the following.
\begin{enumerate}
\item[(a)] If $P \in {\mathcal L}[x_{1},\ldots,x_{n}]$ then $P \in {\rm Span}_{{\mathcal L}}({\rm Tr}(e_{1}P),\ldots,{\rm Tr}(e_{m}P))$.

\item[(b)] If $I<{\mathcal L}[x_{1},\ldots,x_{n}]$ is an ideal so that $P^{\sigma} \in I$ for every 
$\sigma \in \Gamma$ and every $P \in I,$ then $I$ can be generated as ideal by polynomials in $I \cap {\mathcal K}[x_{1},\ldots,x_{n}]$.
\end{enumerate}
\end{lemm}

\begin{proof} For each $j=1, \ldots, m$ we define
$$
Q_{j}:={\rm Tr}(e_{j}P) \in {\mathcal L}[x_{1},\ldots,x_{n}],$$and notice that $Q_{j} \in {\mathcal K}[x_{1},\ldots,x_{n}]$ because $Q_{j}^{\sigma}=Q_{j}$ for each $\sigma \in \Gamma.$

As the matrix $A$ is non-singular, there are  
%
%
%
%the linear transformation $A:{\mathcal L}^{m} \to {\mathcal L}^{m}$ 
%%
%%(where ${\mathcal L}^{n}$ is thought as the vector space of columns of length $m$ and coefficients in ${\mathcal L}$) 
%%
%is invertible. In particular, this ensures 
%
%the existence of  
values $\lambda_{1},\ldots,\lambda_{m} \in {\mathcal L}$ so that
$A\lambda=E$, where
$\lambda=\mbox{}^{t}[\lambda_{1}\; \lambda_{2} \; \ldots \; \lambda_{m}]$ and $E=\mbox{}^{t}[1 \; 0 \; \ldots \; 0]$. In other words, we have $$ \Sigma_{j=1}^m \lambda_j e_j=1 \,\, \mbox{ and } \,\, \Sigma_{j=1}^m \lambda_j\sigma_k(e_j)=0, \,\, \mbox{for each }\,\,  k=2, \ldots, m; $$hence $P$ can be written as $$\lambda_{1}Q_{1}+\cdots+\lambda_{m}Q_{m}=\left(\Sigma_{j=1}^m \lambda_j e_j\right)  P + \Sigma_{k=2}^m \left( P^{\sigma_k}  \Sigma_{j=1}^m \lambda_j \sigma_k(e_j)\right)$$and the first statement is proved. The second statement follows from the first one together with the fact that, if $P^{\sigma} \in I$ for every $\sigma \in \Gamma$ then
$Q_{j} \in I$. 
\end{proof}

%For instance, if in Lemma \ref{polinv} we consider ${\mathcal K}={\mathbb R}$, ${\mathcal L}={\mathbb C}$, $e_{1}=1$ and $e_{2}=i$ then $\lambda_{1}=1/2$, $\lambda_{2}=i/2$ and $P=\frac{1}{2} {\rm Tr}(P) -\frac{i}{2} {\rm Tr}(iP)$.

\begin{lemm}\label{inv}
Let $Y \subset {{\mathcal C}}^{n}$ be an affine algebraic variety defined over ${\mathcal L}$. If $Y^{\sigma}=Y$ for every $\sigma \in \Gamma$  then $Y$ is defined over ${\mathcal K}$. More precisely, if $Y$ is defined by the polynomials $P_{1}, \ldots ,P_{r} \in {\mathcal L}[x_{1},...,x_{n}]$ then $Y$ is also defined by the polynomials ${\rm Tr}(e_{j}P_{i}) \in {\mathcal K}[x_{1}, \ldots, x_{n}]$, where $i \in \{1, \ldots ,r\}$ and $j \in \{1,\ldots ,m\}$.
\end{lemm}
\begin{proof} As $Y$ is defined over ${\mathcal L}$, its associated ideal of polynomials $\mathbb{I}<{\mathcal C}[x_{1},\ldots,x_{n}]$ is generated by a finite collection of polynomials with coefficients in ${\mathcal L}$. Let $P \in \mathbb{I} \cap {\mathcal L}[x_{1},\ldots,x_{n}]$. Let $\sigma \in \Gamma$ and let 
$\eta \in {\rm Gal}({\mathcal C}/{\mathcal K})$ be so that $\rho(\eta)=\sigma$. Then, for $(b_{1},\ldots,b_{n}) \in Y$, it holds that\begin{equation*}
0  = \sigma(P(b_{1}, \ldots ,b_{n}))=\eta(P(b_{1}, \ldots ,b_{n}))    = P^{\eta}(\eta(b_{1}), \ldots ,\eta(b_{n})) = P^{\sigma} \circ \hat{\eta}(b_{1}, \ldots ,b_{n}).
\end{equation*} 
As $Y^{\sigma}=\hat{\eta}(Y)$ and we are assuming $Y^{\sigma}=Y$, we have that  
$\hat{\eta}:Y \to Y$ is a bijection. So, the above asserts that $P^{\sigma}(c_{1}, \ldots ,c_{n})=0$ for each $(c_{1}, \ldots ,c_{n}) \in Y$. This ensures that $P^{\sigma} \in \mathbb{I}$ and the desired result follows directly from Lemma \ref{polinv}.
\end{proof}
%
%\s
%
%We will make also use of the following finiteness result from classical invariant theory.
%Let $V$ be a finite dimensional vector space over a field ${\mathcal K}$, say of dimension $n \geq 1$. Let $x_{1}, \ldots, x_{n}$ be a basis of $V$. The symmetric algebra ${\mathcal K}[V]$ of $V$ can be identified with the free unitary associative algebra generated by $x_{1}$, \ldots, $x_{n}$ over ${\mathcal K}$, that is, with the algebra of polynomials with variables $x_1, \ldots, x_n$ and coefficients in ${\mathcal K}$. If $\Gamma$ is a group acting linearly over $V,$ then that action extends naturally to the diagonal action on ${\mathcal K}[V]$.
%
%
%\begin{theo}[D. Hilbert - E. Noether \cite{Noether1, Noether2}]\label{Hilbert-Noether}
%Let $V$ be a finite dimensional vector space over a field ${\mathcal K}$. If $\Gamma$ is a finite group acting linearly over $V$ then the algebra of $\Gamma-$invariants ${\mathcal K}[V ]^{\Gamma}$ is finitely generated.
%\end{theo}
%
%(see \cite{Noether1, Noether2} and also \cite[Ch. 14]{Procesi})
%
%

%%%%%%%%%%%%%%%%%%%%%%%
%%%%%%%%%%%%%%%%%%%%%%%
\section{Constructive Proof of Weil's Galois  descent theorem}

Let us assume that the algebraic variety $X \subset {\mathcal C}^{n}$ is defined by the  polynomials $P_{1},\ldots , P_{r} \in {\mathcal L}[x_{1}, \ldots ,x_{n}]$ and that we have a Galois descent datum $$\{f_{\sigma_{j}}: X \to X^{\sigma_j}\}_{j=1}^{m}$$for $X$ with respect to the Galois extension ${\mathcal L}/{\mathcal K},$ where $\Gamma={\rm Gal}({\mathcal L}/{\mathcal K})=\{\sigma_{1}=e,\ldots ,\sigma_{m}\}$.

By Lemma \ref{inv}, if $X^{\sigma}=X$ for every $\sigma \in \Gamma$ then $X$ is already defined over ${\mathcal K}$. Thus, from now on, we assume this is not the case.

\begin{lemm}\label{supuesto}
We can assume that $X^{\sigma_{i}} \cap X^{\sigma_{j}}=\emptyset$ for $i \neq j$. 
\end{lemm}

\begin{proof}
For each $\sigma_{j} \neq e$ we may find some $a_{j} \in {\mathcal L}$ so that $\sigma_{j}(a_{j})\neq a_{j}$.  Then we may consider the algebraic variety $\hat{X} \subset {\mathcal C}^{n+m-1}$ defined by the polynomials
$$P_{1}(x_{1},\ldots,x_{n})=\cdots= P_{r}(x_{1}, \ldots ,x_{n})=0, x_{n+1}=a_{2},\ldots, x_{n+m-1}=a_{m}. $$

Clearly, the map $$Q:X \to \hat{X} \,\, \mbox{ given by } \,\, (x_{1},\ldots,x_{n}) \mapsto (x_{1},\ldots,x_{n},a_{2},\ldots,a_{m})$$
defines a biregular isomorphism whose inverse is provided by the projection 
$$Q^{-1}:\hat{X} \to X\,\, \mbox{ given by } \,\, (x_{1},\ldots,x_{n},a_{2},\ldots,a_{m}) \to (x_{1},\ldots,x_{n}).$$

By the construction, we see that $\hat{X}^{\sigma_{i}} \cap \hat{X}^{\sigma_{j}}=\emptyset$ for every $i \neq j$. It is easy to check that a Galois descent datum for $\hat{X}$ is given by
$\{g_{\sigma}=  Q^{\sigma} \circ f_{\sigma} \circ Q^{-1}: \sigma \in \Gamma\}.$
\end{proof}

%By Lemma \ref{supuesto} we assume that,  for $i \neq j$, one has that $X^{\sigma_{i}} \cap X^{\sigma_{j}}=\emptyset$. 

%%%%%%%%%%%%%%%%
\subsection{A first isomorphism}
Let us consider the rational map
$$\Phi:X \to \Pi_{\sigma \in \Gamma} {{\mathcal C}}^{n} \,\, \mbox{ defined by } \,\, x \mapsto \left(f_{\sigma}(x)\right)_{\sigma \in \Gamma}.$$

If $f_{\sigma}(x)=(\tfrac{r_{1,\sigma}(x)}{s_{1,\sigma}(x)}, \ldots , \tfrac{r_{n,\sigma}(x)}{s_{n,\sigma}(x)}),$ 
where $s_{j,\sigma}$ and $r_{j,\sigma}$ are relatively prime polynomials (each one defined over ${\mathcal L}$), then the equality 
$$f_{\sigma}(y_{e})=y_{\sigma}=(y_{1,\sigma}, \ldots, y_{n,\sigma})$$ 
provides $n$ polynomial equations 
$$y_{j,\sigma}s_{j,\sigma}(y_{e})=r_{j,\sigma}(y_{e}).$$ 

We see that $\Phi(X)$ is the affine algebraic variety defined over ${\mathcal L}$ given by 
\begin{equation*}
\begin{aligned}
\Phi(X) ={} &\{ \left(y_{\sigma}\right)_{\sigma \in \Gamma}: P_{1}^{\sigma}(y_{\sigma})=\cdots=P_{r}^{\sigma}(y_{\sigma})=0, \\
      & y_{j,\sigma}s_{j,\sigma}(y_{e})=r_{j,\sigma}(y_{e}),  \sigma \in \Gamma, \; j=1,\ldots, n  \}.
\end{aligned}
\end{equation*}

%\begin{rema}
%By adding some extra variables, one may assume each $f_{\sigma}$ to be a polynomial (for the rest of the process we do not need this fact) and in this case
%$$\Phi(X)=\left\{ \left(y_{\sigma}\right)_{\sigma \in \Gamma}: P_{1}(y_{e})=\cdots=P_{r}(y_{e})=0, y_{\sigma}=f_{\sigma}(y_{e}), \; \sigma \in \Gamma  \right\}.$$
%\end{rema}

\begin{lemm}
The map $\Phi:X \to \Phi(X)$ is a birational isomorphism. Furthermore, $\Phi$ is biregular provided that each $f_{\sigma}$ is polynomial.
\end{lemm}

\begin{proof}
As each $f_{\sigma}$ is a birational isomorphism, $\Phi$ induces a birational isomorphism between $X$ and $\Phi(X)$. The inverse map is given by projection in the first coordinate (so it is regular) 
$$\pi:\Pi_{\sigma \in \Gamma} {{\mathcal C}}^{n} \to {\mathcal C}^{n} \,\, \mbox{ given by } \,\, 
\left(y_{\sigma}\right)_{\sigma \in \Gamma} \mapsto y_{e}.$$ 

If each $f_{\sigma}$ is a polynomial, then $\Phi$ is a regular map and the lemma follows.
\end{proof}

%%%%%%%%%%%%%
\subsection{A linear permutation action on $ \Pi_{\sigma \in \Gamma} {\mathcal C}^{n}$ induced by $\Gamma$}

Let us consider the following natural permutation action of $\Gamma$ on the $\sigma-$coordinates:
$$\Theta:\Gamma \times \Pi_{\sigma \in \Gamma} {\mathcal C}^{n} \to  \Pi_{\sigma \in \Gamma} {\mathcal C}^{n} \,\, \mbox{ given by } \,\, \left(\tau, \left(y_{\sigma}\right)_{\sigma \in \Gamma}\right) \mapsto \left(y_{\tau\sigma}\right)_{\sigma \in \Gamma}$$

\begin{lemm}\label{disjuntos} 
If $\tau \neq e$ then 
$\Theta(\tau)(\Phi(X)) \cap \Phi(X) = \emptyset$.
\end{lemm}
\begin{proof}
Let $(y_{\sigma})_{\sigma \in \Gamma} \in \Theta(\tau)(\Phi(X)) \cap \Phi(X)$. The condition $(y_{\sigma})_{\sigma \in \Gamma} \in \Phi(X)$ ensures that $y_{\sigma} \in X^{\sigma}$ and, in particular,  $y_{\tau^{-1}} \in X^{\tau^{-1}}$. Moreover, the condition 
$(y_{\sigma})_{\sigma \in \Gamma} \in \Theta(\tau)(\Phi(X))$ ensures that $\Theta(\tau^{-1})(y_{\sigma})_{\sigma \in \Gamma}=(y_{\tau^{-1}\sigma})_{\sigma \in \Gamma} \in \Phi(X)$, that is, $y_{\tau^{-1}} \in X$. The above implies that $X \cap X^{\tau^{-1}} \neq \emptyset$, which is not possible by  Lemma \ref{supuesto}. 
\end{proof}

%%%%%%%%%%%%%%%%%%%%
\subsection{A linearization of the Galois descent datum}
Let $\tau \in \Gamma$ and let $\eta \in {\rm Gal}({\mathcal C}/{\mathcal K})$ be so that $\rho(\eta)=\tau$. As $f_{\tau\sigma}=f_{\sigma}^{\tau} \circ f_{\tau}$, we see that
$$\left(\Phi^{\tau} \circ f_{\tau}\right)(x)=\left(f_{\sigma}^{\tau}(f_{\tau}(x))\right)_{\sigma \in \Gamma}=(f_{\tau \sigma}(x))_{\sigma \in \Gamma}=\left(\Theta(\tau) \circ \Phi\right)(x)$$
and, as
$f_{\sigma}^{\tau}=f_{\sigma}^{\eta}=\hat{\eta} \circ f_{\sigma} \circ \hat{\eta}^{\;-1}$, where $\hat{\eta}$ acts on $\Pi_{\sigma \in \Gamma} {\mathcal C}^{n}$,
we obtain the following equality
$$\left(\Phi \circ \hat{\eta}^{\;-1} \circ f_{\tau}\right)(x))=\left(f_{\sigma}(\hat{\eta}^{\;-1}(f_{\tau}(x)))\right)_{\sigma \in \Gamma}=\left(\hat{\eta}^{\;-1}(f^{\tau}_{\sigma}(f_{\tau}(x)))\right)_{\sigma \in \Gamma}=$$
$$=\left(\hat{\eta}^{\;-1}(f_{\tau\sigma}(x))\right)_{\sigma \in \Gamma}=\left(\hat{\eta}^{\;-1} \circ \Theta(\tau) \circ \Phi \right)(x).$$

The above can be summarized in the following commutative diagram.
\begin{equation}\label{diagrama0}
\begin{tikzpicture}[node distance=6.5 cm, auto]
  \node (P) {$X$};
  \node (Q) [right of=P] {$\Phi(X)$};
   \node (A) [below of=P, node distance=1.4 cm] {$X^{\tau}$};
    \node (L) [below of=A, node distance=1.4 cm] {$X$};
  \node (C) [below of=Q, node distance=1.4 cm] {$\Theta(\tau)(\Phi(X))=\Phi^{\tau}(X^{\tau}) =\Phi(X)^{\tau}$};
  \node (M) [below of=C, node distance=1.4 cm] {$\Phi(X)$};
 \draw[->] (L) to node {$\Phi$} (M);
  \draw[->] (A) to node {$\hat{\eta}^{-1}$} (L);
 \draw[->] (C) to node {$\hat{\eta}^{-1}$} (M);
  \draw[->] (P) to node {$\Phi$} (Q);
  \draw[->] (A) to node {$\Phi^{\tau}$} (C);
  \draw[->] (P) to node [swap] {$f_{\tau}$} (A);
    \draw[->] (Q) to node {$\Theta(\tau)$} (C);
\end{tikzpicture}
\end{equation}

Similarly, it is not difficult to see that, for every $\tau \in \Gamma$ and every $\eta \in {\rm Gal}({\mathcal C}/{\mathcal K})$, we have that 
\begin{equation*} \Theta(\tau) \circ \hat{\eta}=\hat{\eta} \circ \Theta(\tau).
\end{equation*}

\begin{rema}
Note that, from the commutative diagram \eqref{diagrama0},  each isomorphism $f_{\sigma}:X \to X^{\sigma}$ is transformed into the linear (permutation) isomorphism $\Theta(\sigma):\Phi(X) \to \Phi(X)^{\sigma}$. So, the Galois descent datum $\{f_{\sigma}\}_{\sigma \in \Gamma}$ for $X$ is now transformed into the (linear) Galois descent datum $\{\Theta(\sigma)\}_{\sigma \in \Gamma}$.  In other words, the above method says that we are able to change our model of $X$ (in an explicit manner) to $W=\Phi(X)$ for which the Galois descent datum is given by permutation linear transformations.
\end{rema}

%%%%%%%%%%%%%%%%
\subsection{A second isomorphism: Invariant theory}

As $\Theta(\Gamma) \cong \Gamma$ is a finite group of permutations, it follows from Hilbert-Noether's theorem (see \cite{Noether1, Noether2} and also \cite[Ch. 14]{Procesi}) that the algebra ${\mathcal C}[\Pi_{\sigma \in \Gamma} {\mathcal C}^{n}]^{\Theta(\Gamma)}$ of $\Theta(\Gamma)$-invariant polynomials with coefficients in ${\mathcal C},$ is finitely generated.

Let us consider a finite set of generators of such an ${\mathcal C}$-algebra, say 
$$E_{1}((y_{\sigma})_{\sigma \in \Gamma}),\ldots,E_{N}((y_{\sigma})_{\sigma \in \Gamma}) \in {\mathcal C}[(y_{\sigma})_{\sigma \in \Gamma}]^{\Theta(\Gamma)}.$$

At this point it is important to note that the finite permutation of coordinates produced by the action of $\Theta(\Gamma)$ does not depend on the field ${\mathcal C}$, that is, we may consider this permutation action on the product space $\Pi_{\sigma \in \Gamma} {\mathcal B}$, where ${\mathcal B}$ is the basic field of ${\mathcal C}$. It follows, in particular, that 
$$E_{1}((y_{\sigma})_{\sigma \in \Gamma}),\ldots, E_{N}((y_{\sigma})_{\sigma \in \Gamma}) \in {\mathcal K}[(y_{\sigma})_{\sigma \in \Gamma}].$$

For the constructiveness part we need to have explicitly computed such polynomials $E_{j}$ (this can be done, for instance, with MAGMA \cite{MAGMA} or Macauley2 \cite{Mac2}). 

Consider the regular map
$$\Psi:\Pi_{\sigma \in \Gamma} {\mathcal C}^{n} \to {\mathcal C}^{N} \,\, \mbox{ given by } \,\, (y_{\sigma})_{\sigma \in \Gamma} \mapsto (E_{1}((y_{\sigma})_{\sigma \in \Gamma}),\ldots, E_{N}((y_{\sigma})_{\sigma \in \Gamma})).$$

\begin{lemm}\label{cubriente}
The regular map $\Psi$ satisfies the following properties:

\begin{enumerate}

\item[(a)] $\Psi^{\sigma}=\Psi$ for every $\sigma \in \Gamma.$
\item[(b)] $\Psi \circ \Theta(\sigma)=\Psi$ for every $\sigma \in \Gamma.$
\item[(c)] if $\Psi(w)=\Psi(z)$ then there is some $\sigma \in \Gamma$ so that $w=\Theta(\sigma)(z).$
\item[(d)] $Y=\Psi(\Phi(X))$ and $\Psi(\Pi_{\sigma \in \Gamma}{\mathcal C}^{n})$ are algebraic subvarieties of ${\mathcal C}^{N}$.
\end{enumerate}
\end{lemm}
\begin{proof}
Properties (a) and (b) are easy to check. 
Property (c) is consequence of the fact that (i) a finite group is a reductive group \cite{Hilbert,Mumford} and (ii) for a reductive group $G$, say acting linearly over ${\mathcal C}^{d}$, and a set of generators of the $G-$invariant polynomials, say $I_{1},\ldots,I_{m} \in {\mathcal C}[z_{1},\ldots,z_{d}]$, the map $A=(I_{1},\ldots,I_{m}):{\mathcal C}^{d} \to {\mathcal C}^{m}$ turns out to be a regular branched cover map with $G$ as its deck group. In other words, $A(p)=A(q)$ if and only if there exists some $\alpha \in G$ so that $\alpha(p)=q$; the branch values of $A$ agree with the $a-$images of those points with non-trivial stabilizer (for details see, for instance, \cite{Hilbert,Mumford}). Property (d) follows from the previous ones.
\end{proof}

\begin{rema}
The algebra of regular maps on $\Psi(\Pi_{\sigma \in \Gamma} {\mathcal C}^{n})$ is known to be isomorphic to the algebra ${\mathcal C}[(y_{\sigma})_{\sigma \in \Gamma}]^{\Theta(\Gamma)}$ of symmetric polynomials with respect to the linear group $\Theta(\Gamma);$ see \cite{CLO}. This can be seen by considering the surjective homomorphism $\xi:{\mathcal C}[t_{1},\ldots,t_{N}] \to {\mathcal C}[(y_{\sigma})_{\sigma \in \Gamma}]^{\Theta(\Gamma)}$ defined by $\xi(p)=p(E_{1},\ldots,E_{N})$.
\end{rema}

Properties (b) and (c) of Lemma \ref{cubriente} assert that $\Psi$ is a finite regular (branched) cover with the finite algebraic group $\Theta(\Gamma)$ as its deck group. 

%The following  result states that its restriction to $\Phi(X)$ induces a biregular isomorphism onto its image algebraic variety $Y=\Psi(\Phi(X))$.

\begin{lemm}
The map $\Psi:\Phi(X) \to Y$ is a biregular isomorphism. In particular, $R=\Psi \circ \Phi:X \to Y$ is a birational isomorphism; it is biregular provided that each $f_{\sigma}$ is polynomial.
\end{lemm}
\begin{proof}
Set $\Phi(X)=W$. Since for each $\tau \in \Gamma-\{e\}$ one has that $\Theta(\tau)(W) \cap W=\emptyset$ (see Lemma \ref{disjuntos}), the polynomial map $\Psi:W \to Y$ is bijective. The set $$\hat{W}:=\cup_{\sigma \in \Gamma} \Theta(\sigma)(W)$$is a reducible affine variety whose (pairwise disjoint) irreducible components are $\Theta(\sigma)(W)$, for $\sigma \in \Gamma$. Since these irreducible components are pairwise disjoint, we may see that the algebra of regular maps on $\hat{W}$, say ${\mathcal C}[\hat{W}]$, is the product of the algebras of regular maps of the components, that is,
$${\mathcal C}[\hat{W}]=\Pi_{\sigma \in \Gamma} {\mathcal C}[\Theta(\sigma)(W)].$$

The above isomorphism is given by the restriction of each regular map of $\hat{W}$ to each of its irreducible components. Note that there is natural isomorphism $$\rho:{\mathcal C}[W] \to {\mathcal C}[\hat{W}]^{\Theta(\Gamma)},$$where ${\mathcal C}[\hat{W}]^{\Theta(\Gamma)}$ denotes the sub-algebra of $\Theta(\Gamma)$-invariant regular maps on $\hat{W}$. This isomorphism is given as follows. If $p \in {\mathcal C}[W]$ then for each $\sigma \in \Gamma$ we may consider the regular map $\rho_{\sigma}(p)=p \circ \Theta(\sigma)^{-1} \in {\mathcal C}[\Theta(\sigma)(W)]$. Then $\rho(p):=( \rho_{\sigma}(p) )_{\sigma \in \Gamma}$ turns out to be an injective homomorphism. It is clear that every $\Theta(\Gamma)-$invariant regular map of $\hat{W}$ is obtained in that way (so $\rho$ is surjective). 

On the other hand, ${\mathcal C}[Y]$ is isomorphic to ${\mathcal C}[\hat{W}]^{\Theta(\Gamma)}$. To see this, one may consider the injective homomorphism $\chi:{\mathcal C}[Y] \to {\mathcal C}[\hat{W}]^{\Theta(\Gamma)}$ defined by $\chi(p)=p \circ \Psi$. Now, to see that $\chi$ is onto we only need to note that $\rho^{-1}(\chi({\mathcal C}[Y]))$ is a sub-algebra of ${\mathcal C}[W]$, that $W$ is irreducible and that $Y$ has the same dimension as $W$.

In this way, $\chi^{-1} \circ \rho$ produces an isomorphism between ${\mathcal C}[W]$ and ${\mathcal C}[Y]$, that is, 
$\Psi:\Phi(X) \to Y$ is a biregular isomorphism and, in particular, $R=\Psi \circ \Phi:X \to Y$ is a birational isomorphism. 

As $\Psi:W \to Y$ is biregular isomorphism and $\Phi:X \to W$ is biregular if each $f_{\sigma}$ is polynomial, then $R:X \to Y$ turns out to be biregular provided that each $f_{\sigma}$ is polynomial. 
\end{proof}

\begin{lemm}
$Y$ is defined over ${\mathcal K}$.
\end{lemm}
\begin{proof}
Let $\tau \in \Gamma$ and let $\eta \in {\rm Gal}({\mathcal C}/{\mathcal K})$ be so that $\rho(\eta)=\tau$. It follows from part (a) of Lemma \ref{cubriente} that 
$$\Psi \circ \hat{\eta}= \hat{\eta} \circ \Psi.$$

Now one sees that the bijection $\hat{\eta}:\Pi_{\sigma \in \Gamma} {\mathcal C}^{n} \to \Pi_{\sigma \in \Gamma} {\mathcal C}^{n}$ descends to the bijection $\hat{\eta}:{{\mathcal C}}^{N} \to {{\mathcal C}}^{N}$. Then it follows from diagram \eqref{diagrama0} that $\hat{\eta}(Y)=Y$, 
that is, $Y^{\tau}=Y$. As this holds for every $\tau \in \Gamma$, it follows from Lemma \ref{inv} that $Y$ is defined over ${\mathcal K}$.
\end{proof}

\begin{rema}
\mbox{}
Above we have constructed an explicit birational isomorphism $R:X \to Y$. We should be able to construct explicit equations for $Y$, say given by polynomials $Q_{1},\ldots, Q_{m} \in {\mathcal L}[t_{1}, \ldots, t_{N}]$ (again, this can be done by using MAGMA).  As we already know that $Y$ is defined over ${\mathcal K}$, so $Y$ is also defined by some polynomials over ${\mathcal K}$. To obtain these polynomials over ${\mathcal K}$, we proceed to replace each of the polynomials $Q_{j}$, which is not already defined over ${\mathcal K}$, by the set of polynomials  ${\rm Tr}(e_{1}Q_{j}),\ldots, {\rm Tr}(e_{m}Q_{j}) \in {\mathcal K}[x_{1},\ldots,x_{n}]$, where $\{e_{1},\ldots,e_{m}\}$ is a basis of ${\mathcal L}$ as ${\mathcal K}-$vector space (see the proof of Lemma \ref{polinv}).
\end{rema}

%%%%%%%%%%%%%%%
%\begin{rema}[Adapting the above to Theorem \ref{teoweil2}]\label{adaptacion}
%In order to see how to adapt the previous constructive proof to obtain a proof of Theorem \ref{teoweil2}, we proceed as follows. We assume $Z \subset {\mathcal C}^{m}$. We consider $$\widetilde{\Phi}: Z  \to \widetilde{Z}:= \widetilde{\Phi}(Z) \subset \prod_{\sigma \in \Gamma} {\mathcal C}^{m} : z  \mapsto (z, \ldots, z).$$We also define $\widetilde{\phi}((y_{\sigma})_{\sigma \in \Gamma})=(\phi^{\sigma}(y_{\sigma}))_{\sigma \in \Gamma}$, where $\Phi(x)=(y_{\sigma})_{\sigma \in \Gamma}$ is as defined in the previous section. If $\Psi:\Phi(X) \to Y$ is as defined in the previous section, then set $N:Y \to \widetilde{Z}$ by the rule $N \circ \Psi=\widetilde{\phi}$. One may sees that $N^{\sigma}=N$, for every $\sigma \in \Gamma$. This follows from the equalities
%\begin{enumerate}
%\item $\Psi^{\sigma}=\Psi$, for every $\sigma \in \Gamma.$
%\item $\Psi \circ \Theta(\sigma)=\Psi$, for every $\sigma \in \Gamma.$
%\item $\tau \circ \widetilde{\phi}=\widetilde{\phi}$, for every permutation $\tau$ of the coordinates.
%\end{enumerate}
%
%Next, if we set $\pi:\prod_{\sigma \in \Gamma} {\mathcal C}^{m} \to {\mathcal C}^{m}$ defined as $\pi((y_{\sigma})_{\sigma \in \Gamma})=y_{e}$, then $L=\pi \circ N:Y \to Z$ is the desired morphism.
%\end{rema}

\section{Applications to complex algebraic varieties and field of moduli}
Weil's Galois descent theorem has also been used in the study of complex algebraic varieties and their fields of moduli; in particular, in the case of Belyi curves and dessin d'enfants (see, for example \cite{Wolfart}).

Let us denote by ${\rm Gal}({\mathbb C})$ the group of field automorphisms of the field of complex numbers and let
$X$ be a complex algebraic variety with a finite group ${\rm Aut}(X)$ of automorphisms. Let $\Gamma_X$ be the subgroup of ${\rm Gal}({\mathbb C})$ consisting of those elements $\sigma$ with the property that $X^{\sigma}$ and $X$ are $\mathbb{C}$-isomorphic. The {\it field of moduli} of $X$, denoted by $\mbox{M}(X)$, is the fixed field of $\Gamma_{X}$. In general, the determination of whether the field of moduli is a field of definition is a difficult task, even in the case of algebraic curves; see, for example \cite{Arte, Earle, Hidalgo, yo1, Huggins, yo2, Shimura, Wolfart}.

\s

%As $\mathbb C$ is algebraically closed of characteristic zero, it holds that ${\rm Fix}({\rm Gal}({\mathbb C}/\mbox{M}(X)))=\mbox{M}(X)$, where ${\rm Gal}({\mathbb C}/\mbox{M}(X))$ denotes the group of field automorphisms of ${\mathbb C}$ acting as the identity over $\mbox{M}(X)$. 

A {\it Galois descent datum} for $X$ is a collection of isomorphisms $f_{\sigma}:X \to X^{\sigma}$ defined over ${\mathbb C}$ such that for every pair $\eta, \tau \in \Gamma_X$ it holds that $f_{\eta \tau}=f_{\tau}^{\eta} \circ f_{\eta}$. We shall suppose that $X$ admits a Galois descent datum and that 
$X$ is defined over a finite Galois extension ${\mathcal L}$ of $\mbox{M}(X)$ (we remark that the second assumption is vacuous in the case of complex algebraic curves; see \cite{HH} and \cite{Koizumi}). Let $\bar{\mathcal L}$ be the algebraic closure of ${\mathcal L}$ in ${\mathbb C}$ (which is also  algebraic closure of $\mbox{M}(X)$). As $\sigma \in \Gamma_{X}$ acts as the identity on $\mbox{M}(X)$, it defines an automorphism of $\bar{\mathcal L}$.  Galois theory implies that $\sigma \in \Gamma_{X}$ acts a field automorphism of $\mathcal L$. Now, as ${\rm Aut}(X)$ is finite, it can be seen that every automorphism of $X$ is defined over $\bar{\mathcal L}$. If $\tau \in {\rm Aut}({\mathbb C}/\bar{\mathcal L})$ then $X^{\tau}=X$ and, in particular, $\tau \in \Gamma_{X}$ and $f_{\tau} \in {\rm Aut}(X)$. 

Let $\sigma \in \Gamma_{X}$. If $\tau \in {\rm Aut}({\mathbb C}/\overline{\mathcal L})$ then $f_{\sigma}^{\tau}:X \to X^{\sigma}$. 
It follows that there is some $h_{\tau} \in {\rm Aut}(X)$ so that $f_{\sigma}^{\tau}=f_{\sigma} \circ h_{\tau}$.
This ensures that $f_{\sigma}$ is defined over $\bar{\mathcal L}$. 

Let us consider the map$$\Theta:{\rm Aut}({\mathbb C}/\bar{\mathcal L}) \to {\rm Aut}(X) \,\, \mbox{ given by } \,\, \sigma \mapsto f_{\sigma}^{-1}.$$

Since, for each $\sigma \in \Gamma_{X}$ the isomorphism $f_{\sigma}$ is defined over $\bar{\mathcal L}$, the condition $f_{\eta \tau}=f_{\tau}^{\eta} \circ f_{\eta}$ ensures that the above is a homomorphism of groups. The kernel of $\Theta$ is a finite index subgroup of ${\rm Aut}({\mathbb C}/\bar{\mathcal L})$ and its fixed field is then a finite extension of $\bar{\mathcal L}$; so equals to $\bar{\mathcal L}$. This implies that the kernel is the whole group ${\rm Aut}({\mathbb C}/\bar{\mathcal L})$. In particular, if $\tau, \sigma \in \Gamma_{X}$ have the same restriction to $\bar{\mathcal L}$ then $f_{\tau}=f_{\sigma}$.

Now, since ${\mathcal L}$ is a finite Galois extension of $\mbox{M}(X)$, we may see that there are only a finite number of algebraic varieties of the form $X^{\tau}$ which are isomorphic to $X$, where $ \tau \in {\rm Gal}({\mathbb C})$. In addition, we have a finite number of possible isomorphisms $f_{\sigma}$ for $\sigma \in \Gamma_{X}$. Thereby, we might assume that $X$ and all these isomorphisms are defined over ${\mathcal L}$.  In this way,  Weil's Galois  descent theorem asserts the following. 

\begin{coro}\label{teoweil0}
Let $X$ be a complex algebraic variety. Assume that ${\rm Aut}(X)$ is finite and that $X$ is defined over a finite extension of its field of moduli $\mbox{M}(X)$. Then
\begin{enumerate}
\item[(a)] $X$ admits a Galois descent datum, say $\{f_{\sigma}\}_{\sigma \in \Gamma_X}$, if and only if  there is  an algebraic variety $Y$ defined over $\mbox{M}(X)$ and a birational isomorphism $R:X \to Y$ defined over the algebraic closure of $\mbox{M}(X)$, such that $R=R^{\sigma} \circ f_{\sigma}$ for every $\sigma \in \Gamma_X$. If, moreover, all the isomorphisms $f_{\sigma}$ are biregular, then $R$ can be chosen to be biregular. 

\item[(b)] If there is another isomorphism $\hat{R}:X \to \hat{Y}$, where $\hat{Y}$ is defined over $\mbox{M}(X)$ and such that $\hat{R}=\hat{R}^{\sigma} \circ f_{\sigma}$, for every $\sigma \in \Gamma_X$ then there exists an isomorphism $J:Y \to \hat{Y}$ defined over $\mbox{M}(X)$ so that $\hat{R}=J \circ R$.

\end{enumerate}
\end{coro}
%%%%%%%%%%%%%%%%%
\section{An Explicit Example: A curve of genus $5$}\label{Sec:Example}
A closed Riemann surface $S$ of genus $5$ admitting a group $H \cong {\mathbb Z}_{2}^{4}$ of conformal automorphisms is called a classical Humbert's curve.  We may identify $S/H$ with the Riemann sphere together with five cone points, which, up to conjugation by a M\"obius transformation, can be supposed to be $\{\infty, 0, 1, \lambda_{1}, \lambda_{2}\}$. In \cite{CGHR} it was proved that $S$ can be represented by the following irreducible and non-singular complex projective curve
\begin{equation*}
\left \{ \begin{array}{lllllll}
\,\,\,\,\,\,u_{1}^{2} & + & u_{2}^{2} & + & u_{3}^{2} & =  & 0\\
\lambda_{1} u_{1}^{2}  & + & u_{2}^{2} & + & u_{4}^{2} & =  & 0\\
\lambda_{2} u_{1}^{2}  & + & u_{2}^{2} & + & u_{5}^{2} & =  & 0\\
\end{array}
\right\}
\subset {\mathbb P}_{\mathbb C}^{4}.
\end{equation*}

We may consider an affine model by taking $u_{1}=1$. For our example we consider $\lambda_{1}=-1$ and $\lambda_{2}=i$. If we set $u_{2}=x_{1}$, $u_{3}=x_{2}$, $u_{4}=x_{3}$ and $u_{5}=x_{4}$, then the affine model (defined over ${\mathbb Q}(i)$) is given by 
\begin{equation*}
X :  \left \{ \begin{array}{lllllll}
\,\,\,\,\,1 & + & x_1^2 & + & x_2^2 & =  & 0\\
 -1  & + & x_1^2 & + & x_3^2 & =  & 0\\
\,\,\, i  & + & x_1^2 & + & x_4^2 & =  & 0\\
\end{array}
\right\}
\subset {\mathbb C}^{4}.
\end{equation*}

Notice that the involution
\begin{equation*}
L:{\mathbb C}^{4} \to {\mathbb C}^{4} \,\, \mbox{ given by }\,\, 
(x_1,x_2,x_3,x_4) \mapsto (-i \;\overline{x_{1}},-i \; \overline{x_{3}},-i \; \overline{x_{2}},-i \;\overline{x_{4}})
\end{equation*}
keeps invariant $X$, so it defines 
a symmetry of $X$. It follows from Weil's Galois descent theorem that $X$ is in fact definable over ${\mathcal K}={\mathbb Q}={\mathbb Q}(i) \cap {\mathbb R}$. Let us consider $\mbox{Gal}(\mathbb{Q}(i)/\mathbb{Q})=\{\sigma_1(z)=z, \sigma_2(z)=\overline{z}\}$ and set 
\begin{equation*}
f_{\sigma_2}:X \to \overline{X}=X^{\sigma_{2}} \,\, \mbox{ given by } \,\,  
(x_1,x_2,x_3,x_4) \mapsto (i x_1,i x_3,i x_2,i x_4).
 \end{equation*}

In this example, 
\begin{equation*}
\Phi : X \subset  \mathbb
{C}^4 \to \Phi (X) \subset  \mathbb{C}^{8}: (x_1,x_2,x_3,x_4) \mapsto (x_1,x_2,x_3,x_4,i x_1,i x_3,i x_2,i x_4).
\end{equation*}
Therefore, the equations defining $\Phi(X)$  are given by
\begin{equation*}
\left\{ \begin{array}{cccccc}
 z_{1}-ix_{1}=0, &  z_{2}-ix_{3}=0 & z_{3}-ix_{2}=0, &  z_{4}-ix_{4}=0\\
 1+x_{1}^2+x_{2}^2=0,& -1+x_{1}^2+x_{3}^2=0,& i+x_{1}^2+x_{4}^2=0 
\end{array}\right\}.
\end{equation*}

We have that  $R:X \to Y=R(X): (x_{1},x_{2},x_{3},x_{4}) \mapsto (t_{1},\ldots,t_{14})$, where
\begin{equation*}\begin{array}{c}
t_{1}=(1+i)x_{1},\;
t_{2}=x_{2}+i x_{3},\;
t_{3}=x_{3}+i x_{2}, \;
t_{4}=(1+i)x_{4},\;
t_{5}=i x_{1}^{2}, \\
t_{6}=x_{2}+ix_{3}, \;
t_{7}=i x_{2}x_{3},\;
t_{8}=i x_{4}^{2}, \;
t_{9}=x_{1}x_{2}-x_{1}x_{3}, \;
t_{10}=x_{1}x_{3}-x_{1}x_{2},\\
t_{11}=0,\;
t_{12}=0,\;
t_{13}=x_{2}x_{4}-x_{3}x_{4},\;
t_{14}=x_{3}x_{4}-x_{2}x_{4}.
\end{array}
\end{equation*}

The inverse map $R^{-1}:Y \to X$ is given as 
\begin{equation*}
(t_{1},\ldots,t_{14}) \mapsto (x_{1},x_{2},x_{3},x_{4})=\left( \frac{t_{1}}{1+i}, \frac{t_{2}-i t_{3}}{2}, \frac{t_{3}-i t_{2}}{2}, \frac{t_{4}}{1+i}\right).
\end{equation*}

Explicit equations for $Y$ are given by 
\begin{equation*} Y:=\left\{\begin{array}{c}
4+t_{2}^{2}-t_{3}^{2}=0,\\
t_{1}^{2}+t_{2}t_{3}=0,\\
t_{1}^{2}+t_{4}^{2}-2=0,\\
t_{14}=-t_{13}=t_{4}(t_{3}-t_{2})/2, \; t_{12}=t_{11}=0, \\
t_{10}=-t_{9}=-t_{1}(t_{2}-t_{3})/2,\; 
t_{8}=t_{4}^2/2,\\
t_{7}=(t_{2}^{2}+t_{3}^{2})/4,\;
t_{6}=t_{2},\;
t_{5} =t_{1}^2/2.
\end{array}
\right\}
\end{equation*}

The above, in particular, asserts that (by forgetting the coordinates $t_{j}$, for $j\geq 5$) that 
the algebraic curve
\begin{equation*}
\hat{Y}=\left\{\begin{array}{ccccc}
4+w_{2}^{2}-w_{3}^{2}=0\\
w_{1}^{2}+w_{2}w_{3}=0\\
 w_{1}^{2}+w_{4}^{2}-2=0
\end{array}\right\} \subset {\mathbb C}^{4}
\end{equation*}
is isomorphic to $X$ by the isomorphism
\begin{equation*}
\hat{R}:X \to \hat{Y} \, \, \mbox{ given by } \,\,(x_{1},x_{2},x_{3},x_{4}) \mapsto (t_{1},t_{2},t_{3},t_{4})=(w_{1},w_{2},w_{3},w_{4}).
\end{equation*}

\begin{rema}\label{rutina}
In this example, the MAGMA routine we can use is the following.
\begin{enumerate}
\item[$>$] $Q$:=Rationals(\;);
\item[$>$] $P<t>$:=PolynomialRing($Q$);
\item[$>$] $q:=t^2+1$;
\item[$>$] $K<r>$:=SplittingField($q$);
\item[$>$]  $A<x_{1},x_{2},x_{3},x_{4}>$:=AffineSpace($K$,$4$);
\item[$>$] $B<t_{1},\ldots,t_{14}>$:=AffineSpace($K$, $14$);
\item[$>$] $X$:=Scheme($A$,$[1+x_{1}^{2}+x_{2}^{2},-1+x_{1}^{2}+x_{3}^{2},i+x_{1}^{2}+x_{4}^{2} ]$);
\item[$>$] $R$:=map$<A->B|[x_{1}+z_{1},
\ldots,x_{3}*x_{4}+z_{3}*z_{4}]>$;
\item[$>$] Image($R$);
\item[$>$] $R(X)$;
\end{enumerate}

With the above routine, MAGMA provides 
equations for $Z$ over ${\mathbb Q}$:
\begin{equation*}
\left\{\begin{array}{ccccc}
t_{14}=-t_{13}=t_{4}(t_{3}-t_{2})/2, \; t_{12}=t_{11}=0\; 
t_{9}+t_{10}=0, \; 
t_{6} - t_{7}=0\;
t_{5} + t_{8} - 1=0,\\
t_{4}^2 - 2 t_{8}=0, \;
t_{3}^2 - 2 t_{7} - 2=0,\;
t_{2}^2 - 2 t_{7} + 2=0,\\
t_{8}^2 t_{10}^2 + t_{8}^2 - 2 t_{8} t_{10}^2 - 2 t_{8} - \frac{1}{4} t_{10}^4 + t_{10}^2 + 1=0,\\
t_{7} t_{10}^2 + t_{8} t_{10}^2 + 2 t_{8} - t_{10}^2 - 2=0,\\
t_{7} t_{8} - t_{7} - t_{8}^2 + 2 t_{8} + \frac{1}{2} t_{10}^2 - 1=0,\\
t_{7}^{2} - t_{8}^2 + 2 t_{8} - 2=0,\;
t_{2} t_{7} + t_{2} - t_{3} t_{8} + t_{3}=0,\\
t_{2} t_{3} - 2 t_{8} + 2=0,\;
t_{1} - 1/2 t_{2} t_{10} -  \frac{1}{2} t_{3} t_{10}=0,\\
t_{2} t_{10}^2 - 2 t_{3} t_{7} + 2 t_{3} t_{8} + t_{3} t_{10}^2=0,\\
t_{2} t_{8} - t_{2} - t_{3} t_{7} + t_{3}=0.
\end{array}\right\}.
\end{equation*}
\end{rema}

\end{document}